\documentclass[12pt]{amsart}
\usepackage[margin=.6in]{geometry}
\usepackage{txfonts}
\usepackage{amsmath}
\usepackage{amsxtra}
\usepackage{amscd}
\usepackage{amsthm}
\usepackage{amssymb}
\usepackage[normalem]{ulem}
\usepackage{comment}
\usepackage{color}
\usepackage{tikz}
\usetikzlibrary{matrix,arrows,decorations.pathmorphing,automata, matrix, positioning, calc, shapes.multipart}
\usepackage{graphicx}
\usepackage{hyperref}
\usepackage[shortlabels]{enumitem}
\usepackage{mynewcommands}
\usepackage{float}
\usepackage{epigraph}

\usepackage{mynewcommands}

\numberwithin{equation}{section}

\theoremstyle{definition}

\usepackage[foot]{amsaddr}

\title{Cylindric multipartitions and level-rank duality}
\author{Thomas Gerber}
\address[]{Lehrstuhl D f\"ur Mathematik, RWTH Aachen, 52062 Aachen, Germany.}
\email{gerber@math.rwth-aachen.de}

\begin{document}

\begin{abstract}
We show that a multipartition is cylindric if and only if its level rank-dual is a source in the corresponding affine type $A$ crystal.
This provides an algebraic interpretation of cylindricity, and completes a similar result for FLOTW multipartitions.
\end{abstract}

\maketitle


\sloppy

\markleft{THOMAS GERBER}
\markright{CYLINDRIC MULTIPARTITIONS AND LEVEL-RANK DUALITY}

\section*{Introduction}

Cylindric multipartitions are a particular class of tuples of partitions 
that were first introduced by Gessel and Krattenthaler \cite{GesselKrattenthaler1997},
in order to study basic hypergeometric series in affine type $A$\footnote{ 
Note that in \cite{GesselKrattenthaler1997}, the terminology used is ``cylindric partitions'',
and these objects are not directly defined as tuples of partitions.
However, it is easy to see them as such, see for instance \cite[Appendix C]{FodaWelsh2015}.
In fact, we will not need Gessel and Krattenthaler's original definition in this paper.
}.
They turned up in representation theory in the work of Foda, Leclerc, Okado, Thibon and Welsh \cite{FLOTW1999},
building on earlier results of Jimbo, Misra, Miwa and Okado \cite{JMMO1991}.
In fact, cylindric multipartitions that verify an additionnal property, called the FLOTW multipartitions,
appear as vertices in the crystal graph of certain irreducible highest weight representations of affine type $A$ quantum groups \cite[Theorem 2.10]{FLOTW1999}.

\medskip

At the combinatorial level, the study of cylindric multipartitions gave rise to many different and interesting results.
A formula for the generating functions for cylindric multipartitions was first given in \cite{GesselKrattenthaler1997} in a special case,
and later by Borodin \cite[Proposition 5.1]{Borodin2007} in the general case, using a probabilistic tool called the periodic Schur process.
Independently, Tingley used the representation-theoretic interpretation of cylindric multipartitions and the Weyl-Kac character formula 
to give an alternative formula for this generating function \cite[Theorem 4.17]{Tingley2008}, 
and showed that it agrees with Borodin's formula.
More recently, Foda and Welsh \cite{FodaWelsh2015} rederived the Andrews-Gordon and Bressoud identities
(generalizing the Rogers-Ramanujan identities) via a systematic study of cylindric multipartitions.
There, they also exploited the relationships with certain characters of the generalized Virasoro algebras.

\medskip

At the representation-theoretical level, the study of cylindric and FLOTW multipartitions have very important applications.
First of all, by Ariki's theorem \cite{Ariki1996}, FLOTW multipartitions parametrize the irreducible representations
of modular cyclotomic Hecke algebras, see \cite[Section 3.4]{FLOTW1999} and further 
investigations \cite{Jacon2004}, \cite{GeckJacon2006}, \cite{Jacon2017}, \cite{GeckJacon2011}.
In fact, one can explicitly construct three commuting crystals on multipartitions
using a combinatorial level-rank duality: two Kac-Moody crystals (of affine type $A$), dual to each other, and a so-called Heisenberg crystal \cite{Gerber2016}.
By the works of Ariki \cite{Ariki2007}, Shan \cite{Shan2011}, Shan and Vasserot \cite{ShanVasserot2012} and Losev \cite{Losev2015},
these crystals are categorified by certain branching rules for representations of cyclotomic Hecke algebras and rational Cherednik algebras (in the category $\cO$).
In particular, sources in the crystals also have an important meaning. 
For instance, finite-dimensional irreducible representations of cyclotomic Cherednik algebras are indexed by multipartitions that are sources in 
the Kac-Moody and Heisenberg crystals simultaneously \cite[Proposition 5.18]{ShanVasserot2012}.
In turn, these correspond to multipartitions whose level-rank dual is FLOTW \cite[Theorem 7.7]{Gerber2016}.

\medskip

In this short note, we complete the latter result by showing that cylindric multipartitions
are precisely those multipartitions whose level-rank dual is a source in the dual Kac-Moody crystal.
We use exclusively combinatorial arguments, relying on the abacus representation of multipartitions,
which we recall in Section \ref{sec_abaci}.
In Section \ref{sec_cyl}, we give the definition of cylindric and FLOTW multipartitions and explain
how to detect cylindricity in an abacus.
Section 3 provides the expected result, namely Theorem \ref{thm_main}.

\section{Combinatorics of abaci}\label{sec_abaci}

In the rest of the paper, we fix $e,\ell\in\Z_{>1}$ and $s\in\Z$. 

\medskip

\subsection{Multipartitions and abaci}

Let $x\in\Z_{>1}$. Later, we will take either $x=\ell$ or $x=e$.
An \textit{$x$-abacus} is a subset $\cA$ of $\Z \times \{ 1,\dots,x \}$
such that there exists $m_-,m_+\in\Z$ verifying: 
\begin{itemize}
 \item For all $\be\leq m_-$ and for all $j\in\{ 1,\dots,x\}$, $(\be,j)\in\cA$.
 \item For all $\be\geq m_+$ and for all $j\in\{ 1,\dots,x\}$, $(\be,j)\notin\cA$.
\end{itemize}

Let us consider the following two different graphical representations of $x$-abaci.
Firstly, we can represent an $x$-abacus $\cA$ by $x$ rows of beads, numbered from $1$ at the bottom to $x$ at the top,
where we put a black (respectively white) bead in column $\be$ and row $j$ if $(\be,j)\in\cA$ (respectively $(\be,j)\notin\cA$).
We call this the horizontal representation of $\cA$.
In a dual fashion, we can represent $\cA$ by $x$ columns of beads, numbered from $1$ at the left to $x$ at the right,
where we put a white (respectively black) bead in row $\be$ and column $j$ if $(\be,j)\in\cA$ (respectively $(\be,j)\notin\cA$).
We call this the vertical representation of $\cA$, and we will use the notation $\dcA$ instead.
In the rest of this paper, we will use the horizontal representation for $\ell$-abaci
and  the vertical representation for $e$-abaci.

Let $\cA$ be an $x$-abacus. The \textit{charge} of $\cA$ is the element $\bs=(s_1,\dots,s_x)\in\Z^x$ such that
in the horizontal representation, the $x$-abacus obtained from $\cA$ by pushing all black beads as far to the left as possible,
the rightmost bead in row $j$, say $(\be,j)$, verifies $\be=s_j$, for all $j\in\{1 \dots, \ell\}$.
We denote $|\bs|=\sum_{i=1}^x s_i$.
Remember that we had fixed $s\in\Z$. We will denote $\Z^x(s)=\left\{ \bs\in\Z^x \mid |\bs|=s \right\}$.

\medskip

\newcommand{\bcA}{\overline{\cA}}

An \textit{$x$-partition} is an $x$-tuple of partitions. 
Denote $\Pi^x$ the set of all $x$-partitions.
Let $\bs=(s_1,\dots,s_x)\in\Z^x$.
The set of $x$-abaci with charge $\bs$ is in bijection with the set of $x$-partitions via the map
$\cA \mapsto \bla=( (\la^{(1)}_1,\la^{(1)}_2,\dots), \dots, (\la^{(x)}_1,\la^{(x)}_2,\dots) )$ defined by 
\begin{equation*}\label{mpab}
(\be,j) \mapsto \la_k^{(j)}=\be-s_j+k-1
\end{equation*}
for all $(\be,j)\in\cA$.

We write $|\bla,\bs\rangle$ for the data consisting of an element $\bs\in\Z^x$ and an $x$-partition $\bla$, 
and call it a \textit{charged $x$-partition}.
Further, we write $\cA=\cA(\bla,\bs)$ for the corresponding $x$-abacus,
and will often identify $\cA$ with $|\bla,\bs\rangle$.

\begin{exa}

Let $\ell=3$, $\bla=(10.9.1 , 9^3.7.6.4.1, .6.3^2)$ be an $\ell$-abacus and $\bs=(-4,0,-5)$ a charge. The horizontal representation of $\cA(\bla,\bs)$ is given in Figure \ref{ab1}.
The dashed line is placed at position $\frac{1}{2}$ in order to keep track of the horizontal grading.
\begin{figure}
\begin{tikzpicture}[scale=0.5, bb/.style={draw,circle,fill,minimum size=2.5mm,inner sep=0pt,outer sep=0pt}, wb/.style={draw,circle,fill=white,minimum size=2.5mm,inner sep=0pt,outer sep=0pt}]

\node [wb] at (11,2) {};
\node [wb] at (10,2) {};
\node [wb] at (9,2) {};
\node [wb] at (8,2) {};
\node [wb] at (7,2) {};
\node [wb] at (6,2) {};
\node [wb] at (5,2) {};
\node [wb] at (4,2) {};
\node [wb] at (3,2) {};
\node [wb] at (2,2) {};
\node [bb] at (1,2) {};
\node [wb] at (0,2) {};
\node [wb] at (-1,2) {};
\node [wb] at (-2,2) {};
\node [bb] at (-3,2) {};
\node [bb] at (-4,2) {};
\node [wb] at (-5,2) {};
\node [wb] at (-6,2) {};
\node [wb] at (-7,2) {};
\node [bb] at (-8,2) {};
\node [bb] at (-9,2) {};

\node [wb] at (11,1) {};
\node [wb] at (10,1) {};
\node [bb] at (9,1) {};
\node [bb] at (8,1) {};
\node [bb] at (7,1) {};
\node [wb] at (6,1) {};
\node [wb] at (5,1) {};
\node [bb] at (4,1) {};
\node [wb] at (3,1) {};
\node [bb] at (2,1) {};
\node [wb] at (1,1) {};
\node [wb] at (0,1) {};
\node [bb] at (-1,1) {};
\node [wb] at (-2,1) {};
\node [wb] at (-3,1) {};
\node [wb] at (-4,1) {};
\node [bb] at (-5,1) {};
\node [wb] at (-6,1) {};
\node [bb] at (-7,1) {};
\node [bb] at (-8,1) {};
\node [bb] at (-9,1) {};

\node [wb] at (11,0) {};
\node [wb] at (10,0) {};
\node [wb] at (9,0) {};
\node [wb] at (8,0) {};
\node [wb] at (7,0) {};
\node [bb] at (6,0) {};
\node [wb] at (5,0) {};
\node [bb] at (4,0) {};
\node [wb] at (3,0) {};
\node [wb] at (2,0) {};
\node [wb] at (1,0) {};
\node [wb] at (0,0) {};
\node [wb] at (-1,0) {};
\node [wb] at (-2,0) {};
\node [wb] at (-3,0) {};
\node [wb] at (-4,0) {};
\node [bb] at (-5,0) {};
\node [wb] at (-6,0) {};
\node [bb] at (-7,0) {};
\node [bb] at (-8,0) {};
\node [bb] at (-9,0) {};

\draw[dashed](0.5,-0.5)--node[]{}(0.5,2.5);

\end{tikzpicture} 
\caption{Drawing the horizontal abacus from a charged multipartition.}
\label{ab1}
\end{figure}
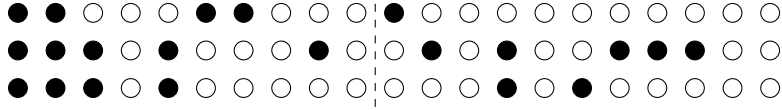

\end{exa}

\subsection{Level-rank duality} 

Let $\cA$ be an $\ell$-abacus with charge $\bs$ and consider the following procedure.
Represent $\cA$ horizontally. 
Stack copies of $\cA$ on top of each other, so that
each new copy is shifted $e$ steps to the right. 
This results in a filling of $\Z^2$ by white and black beads, denoted by $\bcA$.
Choose $e$ consecutive columns of beads such that the index of the rightmost column is divisible by $e$.
This is the vertical representation of a unique abacus, which we denote $\dcA$.
The corresponding charge $\dbs$ verifies $|\dbs|=-|\bs|$.
We denote $\dbla$ the $e$-partition such that $\cA(\dbla,\dbs)=\dcA$.

It is easy to see that the map
\begin{equation*}\label{lr}
\begin{array}{rcl}
   \Pi^\ell\times \Z^\ell(s) & \lra & \Pi^e\times\Z^e(-s) \\
   (\bla,\bs) & \longmapsto & (\dbla,\dbs)
  \end{array}
\end{equation*}
is a bijection. We call it the \textit{level-rank duality}\footnote{Usually, $\ell$ is referred to as the level and $e$ as the rank.}.

\begin{exa}
Take $\ell=4$, $e=3$, $\bla=(3.1,\m,1^3,4.2.1)$ and $\bs=(0,1,1,2)$.
Then the level-rank duality is illustrated in Figure \ref{figlr}.
\begin{figure}
\begin{tikzpicture}[scale=0.5, bb/.style={draw,circle,fill,minimum size=2.5mm,inner sep=0pt,outer sep=0pt}, wb/.style={draw,circle,fill=white,minimum size=2.5mm,inner sep=0pt,outer sep=0pt}]

\node [bb] at (3,-12) {};
\node [bb] at (4,-12) {};
\node [bb] at (5,-12) {};
\node [bb] at (6,-12) {};
\node [bb] at (7,-12) {};
\node [bb] at (8,-12) {};
\node [wb] at (9,-12) {};
\node [bb] at (10,-12) {};
\node [wb] at (11,-12) {};
\node [wb] at (12,-12) {};
\node [bb] at (13,-12) {};
\node [wb] at (14,-12) {};
\node [wb] at (15,-12) {};
\node [wb] at (16,-12) {};
\node [wb] at (17,-12) {};
\node [wb] at (18,-12) {};
\node [wb] at (19,-12) {};

\node [bb] at (3,-11) {};
\node [bb] at (4,-11) {};
\node [bb] at (5,-11) {};
\node [bb] at (6,-11) {};
\node [bb] at (7,-11) {};
\node [bb] at (8,-11) {};
\node [bb] at (9,-11) {};
\node [bb] at (10,-11) {};
\node [bb] at (11,-11) {};
\node [wb] at (12,-11) {};
\node [wb] at (13,-11) {};
\node [wb] at (14,-11) {};
\node [wb] at (15,-11) {};
\node [wb] at (16,-11) {};
\node [wb] at (17,-11) {};
\node [wb] at (18,-11) {};
\node [wb] at (19,-11) {};

\node [bb] at (3,-10) {};
\node [bb] at (4,-10) {};
\node [bb] at (5,-10) {};
\node [bb] at (6,-10) {};
\node [bb] at (7,-10) {};
\node [bb] at (8,-10) {};
\node [wb] at (9,-10) {};
\node [bb] at (10,-10) {};
\node [bb] at (11,-10) {};
\node [bb] at (12,-10) {};
\node [wb] at (13,-10) {};
\node [wb] at (14,-10) {};
\node [wb] at (15,-10) {};
\node [wb] at (16,-10) {};
\node [wb] at (17,-10) {};
\node [wb] at (18,-10) {};
\node [wb] at (19,-10) {};

\node [bb] at (3,-9) {};
\node [bb] at (4,-9) {};
\node [bb] at (5,-9) {};
\node [bb] at (6,-9) {};
\node [bb] at (7,-9) {};
\node [bb] at (8,-9) {};
\node [bb] at (9,-9) {};
\node [wb] at (10,-9) {};
\node [bb] at (11,-9) {};
\node [wb] at (12,-9) {};
\node [bb] at (13,-9) {};
\node [wb] at (14,-9) {};
\node [wb] at (15,-9) {};
\node [bb] at (16,-9) {};
\node [wb] at (17,-9) {};
\node [wb] at (18,-9) {};
\node [wb] at (19,-9) {};

\draw[dashed](10.5,-12.5)--node[]{}(10.5,-8.5);
\draw[dashed](7.5,-12.5)--node[]{}(10.5,-12.5);

\draw[dashed](13.5,-8.5)--node[]{}(13.5,-4.5);
\draw[dashed](10.5,-8.5)--node[]{}(13.5,-8.5);

\node [bb] at (6,-5) {};
\node [bb] at (7,-5) {};
\node [bb] at (8,-5) {};
\node [bb] at (9,-5) {};
\node [bb] at (10,-5) {};
\node [bb] at (11,-5) {};
\node [bb] at (12,-5) {};
\node [wb] at (13,-5) {};
\node [bb] at (14,-5) {};
\node [wb] at (15,-5) {};
\node [bb] at (16,-5) {};
\node [wb] at (17,-5) {};
\node [wb] at (18,-5) {};
\node [bb] at (19,-5) {};
\node [wb] at (20,-5) {};
\node [wb] at (21,-5) {};
\node [wb] at (22,-5) {};
\node [bb] at (6,-6) {};
\node [bb] at (7,-6) {};
\node [bb] at (8,-6) {};
\node [bb] at (9,-6) {};
\node [bb] at (10,-6) {};
\node [bb] at (11,-6) {};
\node [wb] at (12,-6) {};
\node [bb] at (13,-6) {};
\node [bb] at (14,-6) {};
\node [bb] at (15,-6) {};
\node [wb] at (16,-6) {};
\node [wb] at (17,-6) {};
\node [wb] at (18,-6) {};
\node [wb] at (19,-6) {};
\node [wb] at (20,-6) {};
\node [wb] at (21,-6) {};
\node [wb] at (22,-6) {};
\node [bb] at (6,-7) {};
\node [bb] at (7,-7) {};
\node [bb] at (8,-7) {};
\node [bb] at (9,-7) {};
\node [bb] at (10,-7) {};
\node [bb] at (11,-7) {};
\node [bb] at (12,-7) {};
\node [bb] at (13,-7) {};
\node [bb] at (14,-7) {};
\node [wb] at (15,-7) {};
\node [wb] at (16,-7) {};
\node [wb] at (17,-7) {};
\node [wb] at (18,-7) {};
\node [wb] at (19,-7) {};
\node [wb] at (20,-7) {};
\node [wb] at (21,-7) {};
\node [wb] at (22,-7) {};
\node [bb] at (6,-8) {};
\node [bb] at (7,-8) {};
\node [bb] at (8,-8) {};
\node [bb] at (9,-8) {};
\node [bb] at (10,-8) {};
\node [bb] at (11,-8) {};
\node [wb] at (12,-8) {};
\node [bb] at (13,-8) {};
\node [wb] at (14,-8) {};
\node [wb] at (15,-8) {};
\node [bb] at (16,-8) {};
\node [wb] at (17,-8) {};
\node [wb] at (18,-8) {};
\node [wb] at (19,-8) {};
\node [wb] at (20,-8) {};
\node [wb] at (21,-8) {};
\node [wb] at (22,-8) {};

\draw[dashed](16.5,-4.5)--node[]{}(16.5,-0.5);
\draw[dashed](13.5,-4.5)--node[]{}(16.5,-4.5);
\node [bb] at (9,-1) {};
\node [bb] at (10,-1) {};
\node [bb] at (11,-1) {};
\node [bb] at (12,-1) {};
\node [bb] at (13,-1) {};
\node [bb] at (14,-1) {};
\node [bb] at (15,-1) {};
\node [wb] at (16,-1) {};
\node [bb] at (17,-1) {};
\node [wb] at (18,-1) {};
\node [bb] at (19,-1) {};
\node [wb] at (20,-1) {};
\node [wb] at (21,-1) {};
\node [bb] at (22,-1) {};
\node [wb] at (23,-1) {};
\node [wb] at (24,-1) {};
\node [wb] at (25,-1) {};
\node [bb] at (9,-2) {};
\node [bb] at (10,-2) {};
\node [bb] at (11,-2) {};
\node [bb] at (12,-2) {};
\node [bb] at (13,-2) {};
\node [bb] at (14,-2) {};
\node [wb] at (15,-2) {};
\node [bb] at (16,-2) {};
\node [bb] at (17,-2) {};
\node [bb] at (18,-2) {};
\node [wb] at (19,-2) {};
\node [wb] at (20,-2) {};
\node [wb] at (21,-2) {};
\node [wb] at (22,-2) {};
\node [wb] at (23,-2) {};
\node [wb] at (24,-2) {};
\node [wb] at (25,-2) {};
\node [bb] at (9,-3) {};
\node [bb] at (10,-3) {};
\node [bb] at (11,-3) {};
\node [bb] at (12,-3) {};
\node [bb] at (13,-3) {};
\node [bb] at (14,-3) {};
\node [bb] at (15,-3) {};
\node [bb] at (16,-3) {};
\node [bb] at (17,-3) {};
\node [wb] at (18,-3) {};
\node [wb] at (19,-3) {};
\node [wb] at (20,-3) {};
\node [wb] at (21,-3) {};
\node [wb] at (22,-3) {};
\node [wb] at (23,-3) {};
\node [wb] at (24,-3) {};
\node [wb] at (25,-3) {};
\node [bb] at (9,-4) {};
\node [bb] at (10,-4) {};
\node [bb] at (11,-4) {};
\node [bb] at (12,-4) {};
\node [bb] at (13,-4) {};
\node [bb] at (14,-4) {};
\node [wb] at (15,-4) {};
\node [bb] at (16,-4) {};
\node [wb] at (17,-4) {};
\node [wb] at (18,-4) {};
\node [bb] at (19,-4) {};
\node [wb] at (20,-4) {};
\node [wb] at (21,-4) {};
\node [wb] at (22,-4) {};
\node [wb] at (23,-4) {};
\node [wb] at (24,-4) {};
\node [wb] at (25,-4) {};

\draw[dashed](19.5,-0.5)--node[]{}(19.5,3.5);
\draw[dashed](16.5,-0.5)--node[]{}(19.5,-0.5);
\node [bb] at (12,3) {};
\node [bb] at (13,3) {};
\node [bb] at (14,3) {};
\node [bb] at (15,3) {};
\node [bb] at (16,3) {};
\node [bb] at (17,3) {};
\node [bb] at (18,3) {};
\node [wb] at (19,3) {};
\node [bb] at (20,3) {};
\node [wb] at (21,3) {};
\node [bb] at (22,3) {};
\node [wb] at (23,3) {};
\node [wb] at (24,3) {};
\node [bb] at (25,3) {};
\node [wb] at (26,3) {};
\node [wb] at (27,3) {};
\node [wb] at (28,3) {};
\node [bb] at (12,2) {};
\node [bb] at (13,2) {};
\node [bb] at (14,2) {};
\node [bb] at (15,2) {};
\node [bb] at (16,2) {};
\node [bb] at (17,2) {};
\node [wb] at (18,2) {};
\node [bb] at (19,2) {};
\node [bb] at (20,2) {};
\node [bb] at (21,2) {};
\node [wb] at (22,2) {};
\node [wb] at (23,2) {};
\node [wb] at (24,2) {};
\node [wb] at (25,2) {};
\node [wb] at (26,2) {};
\node [wb] at (27,2) {};
\node [wb] at (28,2) {};
\node [bb] at (12,1) {};
\node [bb] at (13,1) {};
\node [bb] at (14,1) {};
\node [bb] at (15,1) {};
\node [bb] at (16,1) {};
\node [bb] at (17,1) {};
\node [bb] at (18,1) {};
\node [bb] at (19,1) {};
\node [bb] at (20,1) {};
\node [wb] at (21,1) {};
\node [wb] at (22,1) {};
\node [wb] at (23,1) {};
\node [wb] at (24,1) {};
\node [wb] at (25,1) {};
\node [wb] at (26,1) {};
\node [wb] at (27,1) {};
\node [wb] at (28,1) {};
\node [bb] at (12,0) {};
\node [bb] at (13,0) {};
\node [bb] at (14,0) {};
\node [bb] at (15,0) {};
\node [bb] at (16,0) {};
\node [bb] at (17,0) {};
\node [wb] at (18,0) {};
\node [bb] at (19,0) {};
\node [wb] at (20,0) {};
\node [wb] at (21,0) {};
\node [bb] at (22,0) {};
\node [wb] at (23,0) {};
\node [wb] at (24,0) {};
\node [wb] at (25,0) {};
\node [wb] at (26,0) {};
\node [wb] at (27,0) {};
\node [wb] at (28,0) {};

\draw[dashed](7.5,-16.5)--node[]{}(7.5,-12.5);
\draw[dashed](4.5,-16.5)--node[]{}(7.5,-16.5);
\node [bb] at (0,-13) {};
\node [bb] at (1,-13) {};
\node [bb] at (2,-13) {};
\node [bb] at (3,-13) {};
\node [bb] at (4,-13) {};
\node [bb] at (5,-13) {};
\node [bb] at (6,-13) {};
\node [wb] at (7,-13) {};
\node [bb] at (8,-13) {};
\node [wb] at (9,-13) {};
\node [bb] at (10,-13) {};
\node [wb] at (11,-13) {};
\node [wb] at (12,-13) {};
\node [bb] at (13,-13) {};
\node [wb] at (14,-13) {};
\node [wb] at (15,-13) {};
\node [wb] at (16,-13) {};
\node [bb] at (0,-14) {};
\node [bb] at (1,-14) {};
\node [bb] at (2,-14) {};
\node [bb] at (3,-14) {};
\node [bb] at (4,-14) {};
\node [bb] at (5,-14) {};
\node [wb] at (6,-14) {};
\node [bb] at (7,-14) {};
\node [bb] at (8,-14) {};
\node [bb] at (9,-14) {};
\node [wb] at (10,-14) {};
\node [wb] at (11,-14) {};
\node [wb] at (12,-14) {};
\node [wb] at (13,-14) {};
\node [wb] at (14,-14) {};
\node [wb] at (15,-14) {};
\node [wb] at (16,-14) {};
\node [bb] at (0,-15) {};
\node [bb] at (1,-15) {};
\node [bb] at (2,-15) {};
\node [bb] at (3,-15) {};
\node [bb] at (4,-15) {};
\node [bb] at (5,-15) {};
\node [bb] at (6,-15) {};
\node [bb] at (7,-15) {};
\node [bb] at (8,-15) {};
\node [wb] at (9,-15) {};
\node [wb] at (10,-15) {};
\node [wb] at (11,-15) {};
\node [wb] at (12,-15) {};
\node [wb] at (13,-15) {};
\node [wb] at (14,-15) {};
\node [wb] at (15,-15) {};
\node [wb] at (16,-15) {};
\node [bb] at (0,-16) {};
\node [bb] at (1,-16) {};
\node [bb] at (2,-16) {};
\node [bb] at (3,-16) {};
\node [bb] at (4,-16) {};
\node [bb] at (5,-16) {};
\node [wb] at (6,-16) {};
\node [bb] at (7,-16) {};
\node [wb] at (8,-16) {};
\node [wb] at (9,-16) {};
\node [bb] at (10,-16) {};
\node [wb] at (11,-16) {};
\node [wb] at (12,-16) {};
\node [wb] at (13,-16) {};
\node [wb] at (14,-16) {};
\node [wb] at (15,-16) {};
\node [wb] at (16,-16) {};

\draw[](16.5,-17)--node[]{}(16.5,4);
\draw[](13.5,-17)--node[]{}(13.5,4);

\node [bb] at (14,3) {};
\node [bb] at (15,3) {};
\node [bb] at (16,3) {};
\node [bb] at (14,2) {};
\node [bb] at (15,2) {};
\node [bb] at (16,2) {};
\node [bb] at (14,1) {};
\node [bb] at (15,1) {};
\node [bb] at (16,1) {};
\node [bb] at (14,0) {};
\node [bb] at (15,0) {};
\node [bb] at (16,0) {};
\node [bb] at (14,-1) {};
\node [bb] at (15,-1) {};
\node [wb] at (16,-1) {};
\node [bb] at (14,-2) {};
\node [wb] at (15,-2) {};
\node [bb] at (16,-2) {};
\node [bb] at (14,-3) {};
\node [bb] at (15,-3) {};
\node [bb] at (16,-3) {};
\node [bb] at (14,-4) {};
\node [wb] at (15,-4) {};
\node [bb] at (16,-4) {};
\draw[dashed](13.5,-4.5)--node[]{}(16.5,-4.5);
\node [bb] at (14,-5) {};
\node [wb] at (15,-5) {};
\node [bb] at (16,-5) {};
\node [bb] at (14,-6) {};
\node [bb] at (15,-6) {};
\node [wb] at (16,-6) {};
\node [bb] at (14,-7) {};
\node [wb] at (15,-7) {};
\node [wb] at (16,-7) {};
\node [wb] at (14,-8) {};
\node [wb] at (15,-8) {};
\node [bb] at (16,-8) {};
\node [wb] at (14,-9) {};
\node [wb] at (15,-9) {};
\node [bb] at (16,-9) {};
\node [wb] at (14,-10) {};
\node [wb] at (15,-10) {};
\node [wb] at (16,-10) {};
\node [wb] at (14,-11) {};
\node [wb] at (15,-11) {};
\node [wb] at (16,-11) {};
\node [wb] at (14,-12) {};
\node [wb] at (15,-12) {};
\node [wb] at (16,-12) {};
\node [wb] at (14,-13) {};
\node [wb] at (15,-13) {};
\node [wb] at (16,-13) {};
\node [wb] at (14,-14) {};
\node [wb] at (15,-14) {};
\node [wb] at (16,-14) {};
\node [wb] at (14,-15) {};
\node [wb] at (15,-15) {};
\node [wb] at (16,-15) {};
\node [wb] at (14,-16) {};
\node [wb] at (15,-16) {};
\node [wb] at (16,-16) {};
\end{tikzpicture}  
\caption{From $\cA$ to $\bcA$ to $\dcA$.}
\label{figlr}
\end{figure}
\end{exa}

\begin{rem}
A slightly simpler version of the combinatorial level-rank duality was introduced by Uglov \cite{Uglov1999},
and used in the works of Tingley \cite{Tingley2008} and Foda and Welsh \cite{FodaWelsh2015}.
Here, we have twisted Uglov's level-rank duality by taking the transpose, in the spirit of \cite{Gerber2016}.
This is essential to prove Theorem \ref{thm_main}, since we use 
the commutation of the crystals \cite[Theorem 4.8]{Gerber2016} that requires this convention.
\end{rem}

\section{Cylindric and FLOTW multipartitions}\label{sec_cyl}

\subsection{Cylindricity and FLOTW property}

Let $e,\ell\in\Z_{>1}$ and $s\in\Z$. 
Set $$D(s)=\left\{ (s_1,\dots,s_\ell)\in\Z^\ell \;\left|\; \sum_{j=1}^\ell s_j = s \mand s_1\leq s_2\leq\dots\leq s_\ell \leq s_1+e \right. \right\}.$$
Note that this set depends on $\ell$ and $e$.
Further, for $\bla=((\la_1^{(1)},\la_2^{(1)},\dots),(\la_1^{(2)},\la_2^{(2)}),\dots,(\la_1^{(\ell)},\la_2^{(\ell)}))\in\Pi^\ell$, $\bs\in\Z^\ell$ and $\al\in\Z_{>0}$, denote 
$$\cR(\bla,\bs,\al)=\left\{ \la_k^{(j)}-k+s_j \mod e \mid \la_k^{(j)}=\al \; ; \; 1\leq j\leq \ell, k\geq 1\right\}.$$

\begin{defi}\label{def_flotw}
Let $\bs\in D(s)$ and $\bla$ be an $\ell$-partition. 
\begin{enumerate}
 \item We say that $|\bla,\bs\rangle$ is \textit{$e$-cylindric} if 
 $\forall 1\leq j\leq \ell-1$,  $\la_k^{(j)}\geq\la_{k+s_{j+1}-s_j}^{(j+1)}\forall k\geq1$ and $\la_k^{(\ell)}\geq\la_{k+e+s_1-s_\ell}^{(1)} \forall k\geq1$.
 \item We say that $\bla$ is \textit{$e$-FLOTW} if the two following conditions hold:
\begin{enumerate}
 \item $|\bla,\bs\rangle$ is $e$-cylindric.
 \item For all $\al\in\Z_{>0}$,  $\cR(\bla,\bs,\al)\neq \{0,\dots,e-1\}$.
\end{enumerate}
\end{enumerate}
\end{defi}

\begin{rem}\label{rem_flotw}
FLOTW multipartitions were introduced in \cite{FLOTW1999}
in order to give a combinatorial realization of irreducible 
highest weight crystals for the quantum group associated to the affine Kac-Moody 
algebra $\sle$,
see \cite[Theorem 2.10]{FLOTW1999}\footnote{In that article, they are called ``highest lifts'', following the original terminology of \cite{JMMO1991}.}.
\end{rem}

\subsection{Yoking beads}\label{sec_yoke}

Cylindricity can be easily detected on an abacus whose charge $\bs$ is in $D(s)$.
Let $\cA'$ be the $(\ell+1)$-abacus defined by $(\be,j)\in\cA' \eq (\be,j)\in\cA$ if $j\leq\ell$ and
$(\be,\ell+1)\in\cA'\eq(\be-e,1)\in\cA$, and consider its horizontal representation.
Simply put, $\cA'$ is just a collection of $\ell+1$ consecutive rows of $\bcA$.

The following procedure is inspired by \cite[Section 3.3]{Tingley2008} and \cite[Section 4]{FodaWelsh2015}.
In $\cA'$, yoke the lefmost white beads in every row to each other, and repeat the procedure recursively on the remaining unyoked white beads.

\begin{exa}\label{exa_yoke}
Take $\ell=4$, $e=3$, $\bla=(3.1,\m,1^3,4.2.1)$ and $\bs=(0,1,1,2)$.
The yoking procedure is illustrated in Figure \ref{yoke1}.
\begin{figure}
\begin{tikzpicture}[scale=0.5, bb/.style={draw,circle,fill,minimum size=2.5mm,inner sep=0pt,outer sep=0pt}, wb/.style={draw,circle,fill=white,minimum size=2.5mm,inner sep=0pt,outer sep=0pt}]

\draw[](12,-8)--node[]{}(10,-9);
\draw[](10,-9)--node[]{}(9,-10);
\draw[](9,-10.1)--node[]{}(12,-10.9);
\draw[](12,-11.1)--node[]{}(9,-11.9);

\draw[](14,-8)--node[]{}(12,-9);
\draw[](12,-9)--node[]{}(13,-10);
\draw[](13,-10)--node[]{}(13,-11);
\draw[](13,-11)--node[]{}(11,-12);

\draw[](15,-8)--node[]{}(14,-9);
\draw[](14,-9)--node[]{}(14,-10);
\draw[](14,-10)--node[]{}(14,-11);
\draw[](14,-11)--node[]{}(12,-12);

\draw[](17,-8)--node[]{}(15,-9);
\draw[](15,-9)--node[]{}(15,-10);
\draw[](15,-10)--node[]{}(15,-11);
\draw[](15,-11)--node[]{}(14,-12);

\draw[](18,-8)--node[]{}(17,-9);
\draw[](17,-9)--node[]{}(16,-10);
\draw[](16,-10)--node[]{}(16,-11);
\draw[](16,-11)--node[]{}(15,-12);

\draw[](19,-8)--node[]{}(18,-9);
\draw[](18,-9)--node[]{}(17,-10);
\draw[](17,-10)--node[]{}(17,-11);
\draw[](17,-11)--node[]{}(16,-12);

\draw[](19,-9)--node[]{}(18,-10);
\draw[](18,-10)--node[]{}(18,-11);
\draw[](18,-11)--node[]{}(17,-12);

\draw[](19,-10)--node[]{}(19,-11);
\draw[](19,-11)--node[]{}(18,-12);

\node [bb] at (6,-12) {};
\node [bb] at (7,-12) {};
\node [bb] at (8,-12) {};
\node [wb] at (9,-12) {};
\node [bb] at (10,-12) {};
\node [wb] at (11,-12) {};
\node [wb] at (12,-12) {};
\node [bb] at (13,-12) {};
\node [wb] at (14,-12) {};
\node [wb] at (15,-12) {};
\node [wb] at (16,-12) {};
\node [wb] at (17,-12) {};
\node [wb] at (18,-12) {};
\node [wb] at (19,-12) {};

\node [bb] at (6,-11) {};
\node [bb] at (7,-11) {};
\node [bb] at (8,-11) {};
\node [bb] at (9,-11) {};
\node [bb] at (10,-11) {};
\node [bb] at (11,-11) {};
\node [wb] at (12,-11) {};
\node [wb] at (13,-11) {};
\node [wb] at (14,-11) {};
\node [wb] at (15,-11) {};
\node [wb] at (16,-11) {};
\node [wb] at (17,-11) {};
\node [wb] at (18,-11) {};
\node [wb] at (19,-11) {};

\node [bb] at (6,-10) {};
\node [bb] at (7,-10) {};
\node [bb] at (8,-10) {};
\node [wb] at (9,-10) {};
\node [bb] at (10,-10) {};
\node [bb] at (11,-10) {};
\node [bb] at (12,-10) {};
\node [wb] at (13,-10) {};
\node [wb] at (14,-10) {};
\node [wb] at (15,-10) {};
\node [wb] at (16,-10) {};
\node [wb] at (17,-10) {};
\node [wb] at (18,-10) {};
\node [wb] at (19,-10) {};

\node [bb] at (6,-9) {};
\node [bb] at (7,-9) {};
\node [bb] at (8,-9) {};
\node [bb] at (9,-9) {};
\node [wb] at (10,-9) {};
\node [bb] at (11,-9) {};
\node [wb] at (12,-9) {};
\node [bb] at (13,-9) {};
\node [wb] at (14,-9) {};
\node [wb] at (15,-9) {};
\node [bb] at (16,-9) {};
\node [wb] at (17,-9) {};
\node [wb] at (18,-9) {};
\node [wb] at (19,-9) {};

\draw[dashed](10.5,-12.5)--node[]{}(10.5,-8.5);

\draw[dashed](13.5,-8.5)--node[]{}(13.5,-7.5);
\draw[dashed](10.5,-8.5)--node[]{}(13.5,-8.5);

%
\node [bb] at (6,-8) {};
\node [bb] at (7,-8) {};
\node [bb] at (8,-8) {};
\node [bb] at (9,-8) {};
\node [bb] at (10,-8) {};
\node [bb] at (11,-8) {};
\node [wb] at (12,-8) {};
\node [bb] at (13,-8) {};
\node [wb] at (14,-8) {};
\node [wb] at (15,-8) {};
\node [bb] at (16,-8) {};
\node [wb] at (17,-8) {};
\node [wb] at (18,-8) {};
\node [wb] at (19,-8) {};

\end{tikzpicture}  
\caption{Yoking white beads in the extended abacus.}
\label{yoke1}
\end{figure}
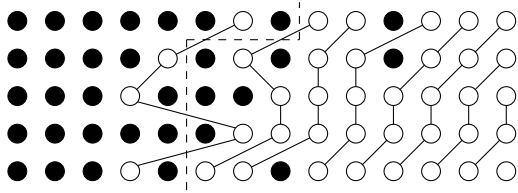
\end{exa}

\begin{lem}\label{lem_yoke}
The abacus $\cA$ is cylindric if and only every pair of yoked white beads $( (\be_1,j+1),(\be_2,j) )$ in $\cA'$ verifies $\be_1\geq\be_2$.
\end{lem}

In other words, an $\ell$-abacus is cylindric if and only if its charge is in $D(s)$ and all yokes have a ``north-east/south-west'' direction.

\begin{proof}
Consider the following dual yoking procedure on $\cA'$.
Let $\be_0$ be the index of the rightmost column of $\cA'$ such that $(\be,j)$ is black for all $\be\leq\be_0$
and for all $j=1,\dots,\ell+1$. 
For all $\be\leq \be_0$, yoke the black beads in consecutive rows of column $\be$ to each other.
By construction, there is a black bead in row $\ell+1$ which is not yoked yet.
Yoke the leftmost such bead to the leftmost black bead in row $\ell$ which is not yoked yet (if it exists).
Iterate for each remaining black bead of row $\ell+1$ of $\cA'$.
Since $\bs\in D(s)$, the only black beads that remain unyoked appear on row $\ell+1$.
Figure \ref{yoke3} illustrates the dual yoking procedure for the abacus of Example \ref{exa_yoke}.
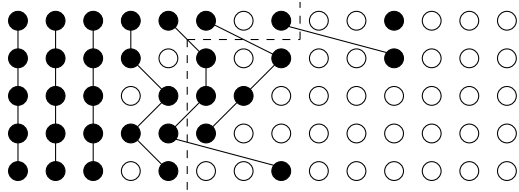
\begin{figure}
\begin{tikzpicture}[scale=0.5, bb/.style={draw,circle,fill,minimum size=2.5mm,inner sep=0pt,outer sep=0pt}, wb/.style={draw,circle,fill=white,minimum size=2.5mm,inner sep=0pt,outer sep=0pt}]

\node [bb] at (6,-12) {};
\node [bb] at (7,-12) {};
\node [bb] at (8,-12) {};
\node [wb] at (9,-12) {};
\node [bb] at (10,-12) {};
\node [wb] at (11,-12) {};
\node [wb] at (12,-12) {};
\node [bb] at (13,-12) {};
\node [wb] at (14,-12) {};
\node [wb] at (15,-12) {};
\node [wb] at (16,-12) {};
\node [wb] at (17,-12) {};
\node [wb] at (18,-12) {};
\node [wb] at (19,-12) {};

\node [bb] at (6,-11) {};
\node [bb] at (7,-11) {};
\node [bb] at (8,-11) {};
\node [bb] at (9,-11) {};
\node [bb] at (10,-11) {};
\node [bb] at (11,-11) {};
\node [wb] at (12,-11) {};
\node [wb] at (13,-11) {};
\node [wb] at (14,-11) {};
\node [wb] at (15,-11) {};
\node [wb] at (16,-11) {};
\node [wb] at (17,-11) {};
\node [wb] at (18,-11) {};
\node [wb] at (19,-11) {};

\node [bb] at (6,-10) {};
\node [bb] at (7,-10) {};
\node [bb] at (8,-10) {};
\node [wb] at (9,-10) {};
\node [bb] at (10,-10) {};
\node [bb] at (11,-10) {};
\node [bb] at (12,-10) {};
\node [wb] at (13,-10) {};
\node [wb] at (14,-10) {};
\node [wb] at (15,-10) {};
\node [wb] at (16,-10) {};
\node [wb] at (17,-10) {};
\node [wb] at (18,-10) {};
\node [wb] at (19,-10) {};

\node [bb] at (6,-9) {};
\node [bb] at (7,-9) {};
\node [bb] at (8,-9) {};
\node [bb] at (9,-9) {};
\node [wb] at (10,-9) {};
\node [bb] at (11,-9) {};
\node [wb] at (12,-9) {};
\node [bb] at (13,-9) {};
\node [wb] at (14,-9) {};
\node [wb] at (15,-9) {};
\node [bb] at (16,-9) {};
\node [wb] at (17,-9) {};
\node [wb] at (18,-9) {};
\node [wb] at (19,-9) {};

\draw[dashed](10.5,-12.5)--node[]{}(10.5,-8.5);

\draw[dashed](13.5,-8.5)--node[]{}(13.5,-7.5);
\draw[dashed](10.5,-8.5)--node[]{}(13.5,-8.5);

\draw[](6,-8)--node[]{}(6,-9);
\draw[](6,-9)--node[]{}(6,-10);
\draw[](6,-10)--node[]{}(6,-11);
\draw[](6,-11)--node[]{}(6,-12);

\draw[](7,-8)--node[]{}(7,-9);
\draw[](7,-9)--node[]{}(7,-10);
\draw[](7,-10)--node[]{}(7,-11);
\draw[](7,-11)--node[]{}(7,-12);

\draw[](8,-8)--node[]{}(8,-9);
\draw[](8,-9)--node[]{}(8,-10);
\draw[](8,-10)--node[]{}(8,-11);
\draw[](8,-11)--node[]{}(8,-12);

\draw[](9,-8)--node[]{}(9,-9);
\draw[](9,-9)--node[]{}(10,-10);
\draw[](10,-10)--node[]{}(9,-11);
\draw[](9,-11)--node[]{}(10,-12);

\draw[](10,-8)--node[]{}(11,-9);
\draw[](11,-9)--node[]{}(11,-10);
\draw[](11,-10)--node[]{}(10,-11);
\draw[](10,-11.1)--node[]{}(13,-11.9);

\draw[](11,-8)--node[]{}(13,-9);
\draw[](13,-9)--node[]{}(12,-10);
\draw[](12,-10)--node[]{}(11,-11);

\draw[](13,-8.1)--node[]{}(16,-8.9);

%
\node [bb] at (6,-8) {};
\node [bb] at (7,-8) {};
\node [bb] at (8,-8) {};
\node [bb] at (9,-8) {};
\node [bb] at (10,-8) {};
\node [bb] at (11,-8) {};
\node [wb] at (12,-8) {};
\node [bb] at (13,-8) {};
\node [wb] at (14,-8) {};
\node [wb] at (15,-8) {};
\node [bb] at (16,-8) {};
\node [wb] at (17,-8) {};
\node [wb] at (18,-8) {};
\node [wb] at (19,-8) {};

\end{tikzpicture}  
\caption{Yoking black beads in the extended abacus.}
\label{yoke3}
\end{figure}
To differentiate the two yoking procedures, let us call a yoke between black beads (respectively white beads) a black yoke (respectively a white yoke).
We claim that $\cA$ is cylindric if and only if every pair of yoked black beads $( (\be_1,j+1),(\be_2,j) )$ in $\cA'$ verifies $\be_1\leq\be_2$.
Indeed, if $1\leq j\leq \ell-1$, we can use the bijection between $\ell$-abaci and $\ell$-partitions given by Formula (\ref{mpab}).
In this case, for yoked beads $( (\be_1,j+1),(\be_2,j) )$ in $\cA'$, we have
\begin{align*}
\be_1 \leq \be_2 
& 
\quad \eq \quad 
\be_1-s_{j+1}+(k+s_{j+1}-s_j)-1 \leq \be_2-s_j+k-1 \\
& \quad \eq \quad 
\la_{k+s_{j+1}-s_j}^{(j+1)} \leq \la_k^{(j)}.
\end{align*}
If $j=\ell$, we have to substract $e$ to $\be_1$ by definition of $\cA'$.
So in this case, we have
\begin{align*}
\be_1 \leq \be_2 
& 
\quad \eq \quad 
\be_1-e+s_{1}+(k+e+s_{1}-s_\ell)-1 \leq \be_2-s_\ell+k-1 \\
& \quad \eq \quad 
\la_{k+e+s_{1}-s_\ell}^{(1)} \leq \la_k^{(\ell)},
\end{align*}
which proves the claim.
In other words, $\cA$ is cylindric if and only if all black yokes have a ``north-west/south-east'' direction.
Finally, we claim that there is a black yoke between $(\be_1,j+1)$ and $(\be_2,j)$ with $\be_1>\be_2$
if and only if there is a white yoke between $(\be'_1,j+1)$ and $(\be'_2,j)$ for some $\be'_1$, $\be'_2$ such that
$\be'_1<\be'_2$.
To see this, assume first that such a black yoke exists and consider the leftmost black yoke verifying this property, with minimal $j$.
Denote $\de=\be_1-\be_2>0$.
Then by minimality, there is a white bead in position $(\be_1-1,j+1)$. Set $\be'_1=\be_1-1$.
This bead is yoke to another white bead (since every white bead belongs to a white yoke) in row $j$.
By definition of the black yokes, there is the same number of black beads to the left of $(\be_1,j+1)$ and to the left of $(\be_2,j)$.
Therefore, by definition of the white yokes, the white bead in row $j$ which is yoked to $(\be'_1,j+1)$ has position $(\be'_2,j)$ with $\be'_2\geq\be_2+\de$.
Thus, $\be'_2\geq \be_2+(\be_1-\be_2)=\be_1=\be'_1+1$, i.e. $\be_2'>\be_1'$.
Conversely, assume there is a white yoke between $(\be'_1,j+1)$ and $(\be'_2,j)$ with $\be'_1<\be'_2$,
and consider the rightmost such yoke, with maximal $j$.
By maximality, there is a black bead in position $(\be'_1+1,j+1)$. Set $\be_1=\be'_1+1$.
If this bead is yoked to another bead, then the same argument as above applies, and the black bead in row $j$ to which it is yoked is $(\be_2,j)$ with $\be_2<\be_1$.
If this bead is not yoked, then it belongs to row $\ell+1$, but then the white bead $(\be'_1,\ell+1)$ is obviously yoked to $(\be'_1-e,1)$.
In turn, all the yokes $(\be',j)$ that connect $(\be'_1,\ell+1)$ and $(\be'_1-e,1)$ verify $\be'<\be'_1$, which contradicts the hypothesis.
\end{proof}

\begin{rem}\label{rem1}
\begin{enumerate}
\item The argument of the above proof show that cylindricity for multipartitions behaves nicely with respect to taking the transpose.
\item In \cite{FodaWelsh2015}, Foda and Welsh use a different convention for representing multipartitions by abaci. One recovers their convention by taking the transpose.
Also, they only consider abaci corresponding to $e$-cylindric multipartitions, and thus do not have an equivalent statement to Lemma \ref{lem_yoke}.
\item In \cite{Tingley2008}, Tingley also uses a different convention, but considers those abaci that essentially verify the combinatorial property of Lemma \ref{lem_yoke}.
He calls them ``descending abaci''. By Lemma \ref{lem_yoke}, descending abaci and cylindric abaci are the same (up to the twist in conventions).
\end{enumerate}
\end{rem}

\begin{exa}
Take $\ell=2$, $e=4$, $\bla=(3^2.1,4.3.2)$ and $\bs=(1,2)$.
Then $|\bla,\bs\rangle$ is cylindric, as the yoking procedure in Figure \ref{yoke2} shows.
\begin{figure}
\begin{tikzpicture}[scale=0.5, bb/.style={draw,circle,fill,minimum size=2.5mm,inner sep=0pt,outer sep=0pt}, wb/.style={draw,circle,fill=white,minimum size=2.5mm,inner sep=0pt,outer sep=0pt}]

\draw[](13,-8.1)--node[]{}(10,-8.9);
\draw[](10,-9)--node[]{}(9,-10);

\draw[](15,-8.2)--node[]{}(11,-8.8);
\draw[](11,-9)--node[]{}(11,-10);

\draw[](16,-8.1)--node[]{}(13,-8.9);
\draw[](13,-9)--node[]{}(12,-10);

\draw[](19,-8.2)--node[]{}(15,-8.8);
\draw[](15,-9)--node[]{}(15,-10);

\draw[](17,-9)--node[]{}(16,-10);

\draw[](18,-9)--node[]{}(17,-10);

\draw[](19,-9)--node[]{}(18,-10);

\node [bb] at (6,-10) {};
\node [bb] at (7,-10) {};
\node [bb] at (8,-10) {};
\node [wb] at (9,-10) {};
\node [bb] at (10,-10) {};
\node [wb] at (11,-10) {};
\node [wb] at (12,-10) {};
\node [bb] at (13,-10) {};
\node [bb] at (14,-10) {};
\node [wb] at (15,-10) {};
\node [wb] at (16,-10) {};
\node [wb] at (17,-10) {};
\node [wb] at (18,-10) {};
\node [wb] at (19,-10) {};

\node [bb] at (6,-9) {};
\node [bb] at (7,-9) {};
\node [bb] at (8,-9) {};
\node [bb] at (9,-9) {};
\node [wb] at (10,-9) {};
\node [wb] at (11,-9) {};
\node [bb] at (12,-9) {};
\node [wb] at (13,-9) {};
\node [bb] at (14,-9) {};
\node [wb] at (15,-9) {};
\node [bb] at (16,-9) {};
\node [wb] at (17,-9) {};
\node [wb] at (18,-9) {};
\node [wb] at (19,-9) {};

\draw[dashed](10.5,-10.5)--node[]{}(10.5,-8.5);

\draw[dashed](14.5,-8.5)--node[]{}(14.5,-7.5);
\draw[dashed](10.5,-8.5)--node[]{}(14.5,-8.5);

%
\node [bb] at (6,-8) {};
\node [bb] at (7,-8) {};
\node [bb] at (8,-8) {};
\node [bb] at (9,-8) {};
\node [bb] at (10,-8) {};
\node [bb] at (11,-8) {};
\node [bb] at (12,-8) {};
\node [wb] at (13,-8) {};
\node [bb] at (14,-8) {};
\node [wb] at (15,-8) {};
\node [wb] at (16,-8) {};
\node [bb] at (17,-8) {};
\node [bb] at (18,-8) {};
\node [wb] at (19,-8) {};

\end{tikzpicture}  
\caption{The result of the yoking procedure: the abacus is cylindric.}
\label{yoke2}
\end{figure}
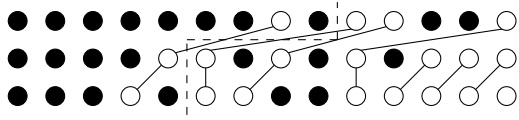
\end{exa}

\section{Crystal characterizations}\label{sec_crys}

Following \cite{Gerber2016}, there are three commuting crystal structures on the set of $\ell$-partitions:
\begin{itemize}
 \item an $\sle$-crystal arising from the integrable action of the quantum group $\Ue$ on the $v$-deformed level $\ell$ Fock space of \cite{JMMO1991},
\item an $\cH$-crystal arising from the action of the quantum Heisenberg algebra on the same space \cite{Uglov1999}, \cite{Gerber2016a},
\item an $\sll$-crystal arising from the integrable action of $\Ul$ on an appropriate direct sum of $v$-deformed level $\ell$ Fock spaces via level-rank duality (\ref{lr}) 
\cite{Uglov1999}, \cite{Gerber2016}.
\end{itemize}
To define these different module structures, one has to fix a charge parameter $\bs\in\Z^\ell$.
Without restriction, we choose $\bs\in\Z^\ell(s)$.
We do not recall here the explicit formulas for computing the different crystals.
For the $\sle$-crystal, we refer to \cite[Chapter 6]{GeckJacon2011}.
The case of $\sll$-crystal is given by the same formula, after switching the role of $\ell$ and $e$ and using Correspondence (\ref{lr}).
For the $\cH$-crystal, the complete explicit formulas have been given recently in \cite{GerberNorton2017}.

\medskip

These crystals are certains oriented colored graphs,
whose vertices are the $\ell$-partitions, and each of whose
connected components have a unique source vertex.
Sources in the $\sle$-crystal have a simple characterization.
In order to state it, we recall the notion of periods in an abacus.
The first \textit{$e$-period} in $\cA$ is, if it exists, the sequence $P=((\be_1,j_1), \dots, (\be_e,j_e))$
of $e$ elements in $\cA$ such that
\begin{itemize}
 \item $\be_1=\max\left\{  \, \be \, | \, (\be,j)\in\cA \text{ for some } j \, \right\}$
 \item $\be_i = \be_{i-1}-1$ for all $i=2,\dots,e$,
 \item $j_i \leq j_{i-1}$ for all $i=2,\dots,e$,
 \item for all $i=1,\dots,e$, there does not exist $(j_0,\be_i)\in\cA$ such that $j_0\leq j_i$.
\end{itemize}
The first period of $\cA \backslash P$, if it exists, is called the second period of $\cA$.
We define similarly the $k$-th period of $\cA$ by induction.
An $\ell$-abacus is called totally $e$-periodic if it has 
infinitely many $e$-periods. 
The following result was proved by Jacon and Lecouvey \cite[Theorem 5.9]{JaconLecouvey2012}.

\begin{thm}\label{thm_totper}
An $\ell$-abacus is a source in the $\sle$-crystal if and only if it is totally $e$-periodic.
\end{thm}

An analogous result for the $\cH$-crystal is given in \cite[Theorem 4.15 and Example 4.18]{GerberNorton2017}.
We will not need it here.

\begin{exa}\label{exa_totper}

Let $e=\ell=3$, $\bla=(1,2^2.1^2,2^4)$ and $\bs=(1,2,3)$.
The abacus $\cA(\bla,\bs)$ is represented horizontally in Figure \ref{ab_totper}.
We have yoked black beads belonging to the same $e$-period. One sees that $\cA(\bla,\bs)$ is totally $e$-periodic.
\begin{figure}
\begin{tikzpicture}[scale=0.5, bb/.style={draw,circle,fill,minimum size=2.5mm,inner sep=0pt,outer sep=0pt}, wb/.style={draw,circle,fill=white,minimum size=2.5mm,inner sep=0pt,outer sep=0pt}]

\node [bb] at (-5,0) {};
\node [bb] at (-4,0) {};
\node [bb] at (-3,0) {};
\node [bb] at (-2,0) {};
\node [bb] at (-1,0) {};
\node [bb] at (0,0) {};
\node [wb] at (1,0) {};
\node [bb] at (2,0) {};
\node [wb] at (3,0) {};
\node [wb] at (4,0) {};
\node [wb] at (5,0) {};
\node [wb] at (6,0) {};
\node [wb] at (7,0) {};
\node [wb] at (8,0) {};
\node [wb] at (9,0) {};

\node [bb] at (-5,1) {};
\node [bb] at (-4,1) {};
\node [bb] at (-3,1) {};
\node [bb] at (-2,1) {};
\node [wb] at (-1,1) {};
\node [bb] at (0,1) {};
\node [bb] at (1,1) {};
\node [wb] at (2,1) {};
\node [bb] at (3,1) {};
\node [bb] at (4,1) {};
\node [wb] at (5,1) {};
\node [wb] at (6,1) {};
\node [wb] at (7,1) {};
\node [wb] at (8,1) {};
\node [wb] at (9,1) {};

\node [bb] at (-5,2) {};
\node [bb] at (-4,2) {};
\node [bb] at (-3,2) {};
\node [bb] at (-2,2) {};
\node [bb] at (-1,2) {};
\node [wb] at (0,2) {};
\node [wb] at (1,2) {};
\node [bb] at (2,2) {};
\node [bb] at (3,2) {};
\node [bb] at (4,2) {};
\node [bb] at (5,2) {};
\node [wb] at (6,2) {};
\node [wb] at (7,2) {};
\node [wb] at (8,2) {};
\node [wb] at (9,2) {};

\draw[dashed](0.5,-0.5)--node[]{}(0.5,2.6);

\draw[](5,2)--node[]{}(4,1);
\draw[](4,1)--node[]{}(3,1);

\draw[](4,2)--node[]{}(3,2);
\draw[](3,2)--node[]{}(2,0);

\draw[](2,2)--node[]{}(1,1);
\draw[](1,1)--node[]{}(0,0);

\draw[](0,1)--node[]{}(-1,0);
\draw[](-1,0)--node[]{}(-2,0);

\draw[](-1,2)--node[]{}(-2,1);
\draw[](-2,1)--node[]{}(-3,0);

\draw[](-2,2)--node[]{}(-3,1);
\draw[](-3,1)--node[]{}(-4,0);

\draw[](-3,2)--node[]{}(-4,1);
\draw[](-4,1)--node[]{}(-5,0);

\draw[](-4,2)--node[]{}(-5,1);

\end{tikzpicture} 
\caption{A totally periodic abacus.}
\label{ab_totper}
\end{figure}
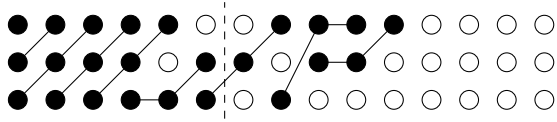

\end{exa}

\medskip

We are ready to prove the following result relating cylindricity (respectively FLOTW property) and sources in crystals.

\begin{thm}\label{thm_main} Let $\cA$ be an $\ell$-abacus.
\begin{enumerate}
 \item $\cA$ is $e$-cylindric if and only if $\dcA$ is a source in the $\sll$-crystal.
 \item $\cA$ is $e$-FLOTW if and only if $\dcA$ is a source in the $\sll$-crystal and in the $\cH$-crystal.
\end{enumerate}
\end{thm}

\begin{proof}
\begin{enumerate}
\item Assume $\cA$ is $e$-cylindric and extend the yoking procedure of Section \ref{sec_yoke} to $\bcA$.
By Lemma \ref{lem_yoke}, all white yokes have a ``north-east/south-west'' direction, 
so that every yoke in $\cA$ spreads on at most $e$ columns.
In fact, slicing $\bcA$ horizontally and then vertically shows that 
every white yoke in $\cA$ corresponds to a period in $\dcA$.
Therefore, $\dcA$ is totally periodic, and the result follows by Theorem \ref{thm_totper}.
Conversely, if one represents a totally periodic abacus vertically, periods
corresponds to yokes in the horizontal level-rank dual which all have a ``north-east/south-west'' direction.
\item This is a direct consequence of \cite[Theorem 6.19]{Gerber2016} and has been shown in \cite[Proof of Theorem 7.7]{Gerber2016} already.
Alternatively, one can use the combinatorial characterizations of (1) and \cite[Theorem 4.15]{GerberNorton2017} to recover this result directly.
\end{enumerate}
\end{proof}

\begin{exa}
Let $e=\ell=3$, $\bla=(3.2,4.2,4)$ and $\bs=(-3,-2,-1)$.
Then, $|\bla,\bs\rangle$ is cylindric as Figure \ref{ab_thm}  shows ($\cA:=\cA(\bla,\bs)$ is represented in $\bcA$).
If we represent the white yokes in $\dcA$, each of them forms an $\ell$-period.
Thus, $|\dbla,\dbs\rangle$ is totally $\ell$-periodic.
In fact, $\dbla=(1,2^2.1^2,2^4)$ and $\dbs=(1,2,3)$, which we had represented horizontally in Example \ref{exa_totper}.
\begin{figure}
\begin{tikzpicture}[scale=0.5, bb/.style={draw,circle,fill,minimum size=2.5mm,inner sep=0pt,outer sep=0pt}, wb/.style={draw,circle,fill=white,minimum size=2.5mm,inner sep=0pt,outer sep=0pt}]

\draw[](-4,0)--node[]{}(-3,1);
\draw[](-3,1)--node[]{}(-1,2);

\draw[](-3,0)--node[]{}(-2,1);
\draw[](-2,1)--node[]{}(0,2);

\draw[](-1,0)--node[]{}(0,1);
\draw[](0,1)--node[]{}(1,2);

\draw[](1,0)--node[]{}(1,1);
\draw[](1,1)--node[]{}(2,2);

\draw[](2,0)--node[]{}(3,1);
\draw[](3,1)--node[]{}(4,2);

\draw[](3,0)--node[]{}(4,1);
\draw[](4,1)--node[]{}(5,2);

\draw[](4,0)--node[]{}(5,1);
\draw[](5,1)--node[]{}(6,2);

\draw[](5,0)--node[]{}(6,1);
\draw[](6,1)--node[]{}(7,2);

\draw[](6,0)--node[]{}(7,1);
\draw[](7,1)--node[]{}(8,2);

\draw[](7,0)--node[]{}(8,1);
\draw[](8,1)--node[]{}(9,2);

\draw[](8,0)--node[]{}(9,1);


\draw[](3,7)--node[]{}(2,6);
\draw[](2,6)--node[]{}(2,5);

\draw[](3,6)--node[]{}(3,5);
\draw[](3,5)--node[]{}(1,4);

\draw[](3,4)--node[]{}(2,3);
\draw[](2,3)--node[]{}(1,2);

\draw[](2,0)--node[]{}(1,-1);

\draw[](3,0)--node[]{}(2,-1);
\draw[](2,-1)--node[]{}(1,-2);

\draw[](3,-1)--node[]{}(2,-2);
\draw[](2,-2)--node[]{}(1,-3);

\draw[](3,-2)--node[]{}(2,-3);

\node [bb] at (-8,0) {};
\node [bb] at (-7,0) {};
\node [bb] at (-6,0) {};
\node [bb] at (-5,0) {};
\node [wb] at (-4,0) {};
\node [wb] at (-3,0) {};
\node [bb] at (-2,0) {};
\node [wb] at (-1,0) {};
\node [bb] at (0,0) {};
\node [wb] at (1,0) {};
\node [wb] at (2,0) {};
\node [wb] at (3,0) {};
\node [wb] at (4,0) {};
\node [wb] at (5,0) {};
\node [wb] at (6,0) {};
\node [wb] at (7,0) {};
\node [wb] at (8,0) {};
\node [wb] at (9,0) {};

\node [bb] at (-8,1) {};
\node [bb] at (-7,1) {};
\node [bb] at (-6,1) {};
\node [bb] at (-5,1) {};
\node [bb] at (-4,1) {};
\node [wb] at (-3,1) {};
\node [wb] at (-2,1) {};
\node [bb] at (-1,1) {};
\node [wb] at (0,1) {};
\node [wb] at (1,1) {};
\node [bb] at (2,1) {};
\node [wb] at (3,1) {};
\node [wb] at (4,1) {};
\node [wb] at (5,1) {};
\node [wb] at (6,1) {};
\node [wb] at (7,1) {};
\node [wb] at (8,1) {};
\node [wb] at (9,1) {};

\node [bb] at (-8,2) {};
\node [bb] at (-7,2) {};
\node [bb] at (-6,2) {};
\node [bb] at (-5,2) {};
\node [bb] at (-4,2) {};
\node [bb] at (-3,2) {};
\node [bb] at (-2,2) {};
\node [wb] at (-1,2) {};
\node [wb] at (0,2) {};
\node [wb] at (1,2) {};
\node [wb] at (2,2) {};
\node [bb] at (3,2) {};
\node [wb] at (4,2) {};
\node [wb] at (5,2) {};
\node [wb] at (6,2) {};
\node [wb] at (7,2) {};
\node [wb] at (8,2) {};
\node [wb] at (9,2) {};

\draw[dashed](0.5,-0.5)--node[]{}(0.5,2.5);


\node [bb] at (-5,3) {};
\node [bb] at (-4,3) {};
\node [bb] at (-3,3) {};
\node [bb] at (-2,3) {};
\node [wb] at (-1,3) {};
\node [wb] at (0,3) {};
\node [bb] at (1,3) {};
\node [wb] at (2,3) {};
\node [bb] at (3,3) {};
\node [wb] at (4,3) {};
\node [wb] at (5,3) {};
\node [wb] at (6,3) {};
\node [wb] at (7,3) {};
\node [wb] at (8,3) {};
\node [wb] at (9,3) {};
\node [wb] at (10,3) {};
\node [wb] at (11,3) {};
\node [wb] at (12,3) {};

\node [bb] at (-5,4) {};
\node [bb] at (-4,4) {};
\node [bb] at (-3,4) {};
\node [bb] at (-2,4) {};
\node [bb] at (-1,4) {};
\node [wb] at (0,4) {};
\node [wb] at (1,4) {};
\node [bb] at (2,4) {};
\node [wb] at (3,4) {};
\node [wb] at (4,4) {};
\node [bb] at (5,4) {};
\node [wb] at (6,4) {};
\node [wb] at (7,4) {};
\node [wb] at (8,4) {};
\node [wb] at (9,4) {};
\node [wb] at (10,4) {};
\node [wb] at (11,4) {};
\node [wb] at (12,4) {};

\node [bb] at (-5,5) {};
\node [bb] at (-4,5) {};
\node [bb] at (-3,5) {};
\node [bb] at (-2,5) {};
\node [bb] at (-1,5) {};
\node [bb] at (0,5) {};
\node [bb] at (1,5) {};
\node [wb] at (2,5) {};
\node [wb] at (3,5) {};
\node [wb] at (4,5) {};
\node [wb] at (5,5) {};
\node [bb] at (6,5) {};
\node [wb] at (7,5) {};
\node [wb] at (8,5) {};
\node [wb] at (9,5) {};
\node [wb] at (10,5) {};
\node [wb] at (11,5) {};
\node [wb] at (12,5) {};

\draw[dashed](3.5,2.5)--node[]{}(3.5,5.5);


\node [bb] at (-2,6) {};
\node [bb] at (-1,6) {};
\node [bb] at (0,6) {};
\node [bb] at (1,6) {};
\node [wb] at (2,6) {};
\node [wb] at (3,6) {};
\node [bb] at (4,6) {};
\node [wb] at (5,6) {};
\node [bb] at (6,6) {};
\node [wb] at (7,6) {};
\node [wb] at (8,6) {};
\node [wb] at (9,6) {};
\node [wb] at (10,6) {};
\node [wb] at (11,6) {};
\node [wb] at (12,6) {};
\node [wb] at (13,6) {};
\node [wb] at (14,6) {};
\node [wb] at (15,6) {};

\node [bb] at (-2,7) {};
\node [bb] at (-1,7) {};
\node [bb] at (-0,7) {};
\node [bb] at (1,7) {};
\node [bb] at (2,7) {};
\node [wb] at (3,7) {};
\node [wb] at (4,7) {};
\node [bb] at (5,7) {};
\node [wb] at (6,7) {};
\node [wb] at (7,7) {};
\node [bb] at (8,7) {};
\node [wb] at (9,7) {};
\node [wb] at (10,7) {};
\node [wb] at (11,7) {};
\node [wb] at (12,7) {};
\node [wb] at (13,7) {};
\node [wb] at (14,7) {};
\node [wb] at (15,7) {};

\node [bb] at (-2,8) {};
\node [bb] at (-1,8) {};
\node [bb] at (0,8) {};
\node [bb] at (1,8) {};
\node [bb] at (2,8) {};
\node [bb] at (3,8) {};
\node [bb] at (4,8) {};
\node [wb] at (5,8) {};
\node [wb] at (6,8) {};
\node [wb] at (7,8) {};
\node [wb] at (8,8) {};
\node [bb] at (9,8) {};
\node [wb] at (10,8) {};
\node [wb] at (11,8) {};
\node [wb] at (12,8) {};
\node [wb] at (13,8) {};
\node [wb] at (14,8) {};
\node [wb] at (15,8) {};

\draw[dashed](6.5,5.5)--node[]{}(6.5,8.5);


\node [bb] at (-11,-3) {};
\node [bb] at (-10,-3) {};
\node [bb] at (-9,-3) {};
\node [bb] at (-8,-3) {};
\node [wb] at (-7,-3) {};
\node [wb] at (-6,-3) {};
\node [bb] at (-5,-3) {};
\node [wb] at (-4,-3) {};
\node [bb] at (-3,-3) {};
\node [wb] at (-2,-3) {};
\node [wb] at (-1,-3) {};
\node [wb] at (0,-3) {};
\node [wb] at (1,-3) {};
\node [wb] at (2,-3) {};
\node [wb] at (3,-3) {};
\node [wb] at (4,-3) {};
\node [wb] at (5,-3) {};
\node [wb] at (6,-3) {};

\node [bb] at (-11,-2) {};
\node [bb] at (-10,-2) {};
\node [bb] at (-9,-2) {};
\node [bb] at (-8,-2) {};
\node [bb] at (-7,-2) {};
\node [wb] at (-6,-2) {};
\node [wb] at (-5,-2) {};
\node [bb] at (-4,-2) {};
\node [wb] at (-3,-2) {};
\node [wb] at (-2,-2) {};
\node [bb] at (-1,-2) {};
\node [wb] at (0,-2) {};
\node [wb] at (1,-2) {};
\node [wb] at (2,-2) {};
\node [wb] at (3,-2) {};
\node [wb] at (4,-2) {};
\node [wb] at (5,-2) {};
\node [wb] at (6,-2) {};

\node [bb] at (-11,-1) {};
\node [bb] at (-10,-1) {};
\node [bb] at (-9,-1) {};
\node [bb] at (-8,-1) {};
\node [bb] at (-7,-1) {};
\node [bb] at (-6,-1) {};
\node [bb] at (-5,-1) {};
\node [wb] at (-4,-1) {};
\node [wb] at (-3,-1) {};
\node [wb] at (-2,-1) {};
\node [wb] at (-1,-1) {};
\node [bb] at (0,-1) {};
\node [wb] at (1,-1) {};
\node [wb] at (2,-1) {};
\node [wb] at (3,-1) {};
\node [wb] at (4,-1) {};
\node [wb] at (5,-1) {};
\node [wb] at (6,-1) {};

\draw[dashed](-2.5,-3.5)--node[]{}(-2.5,-0.5);

\draw[dashed](-2.5,-0.5)--node[]{}(0.5,-0.5);
\draw[dashed](0.5,2.5)--node[]{}(3.5,2.5);
\draw[dashed](3.5,5.5)--node[]{}(6.5,5.5);

\draw[](-9,-0.5)--node[]{}(10,-0.5);
\draw[](-9,2.5)--node[]{}(10,2.5);

\draw[](0.5,9)--node[]{}(0.5,-4);
\draw[](3.5,9)--node[]{}(3.5,-4);

\end{tikzpicture} 
\caption{The level-rank dual of an $e$-cylindric $\ell$-abacus is totally $\ell$-periodic.}
\label{ab_thm}
\end{figure}
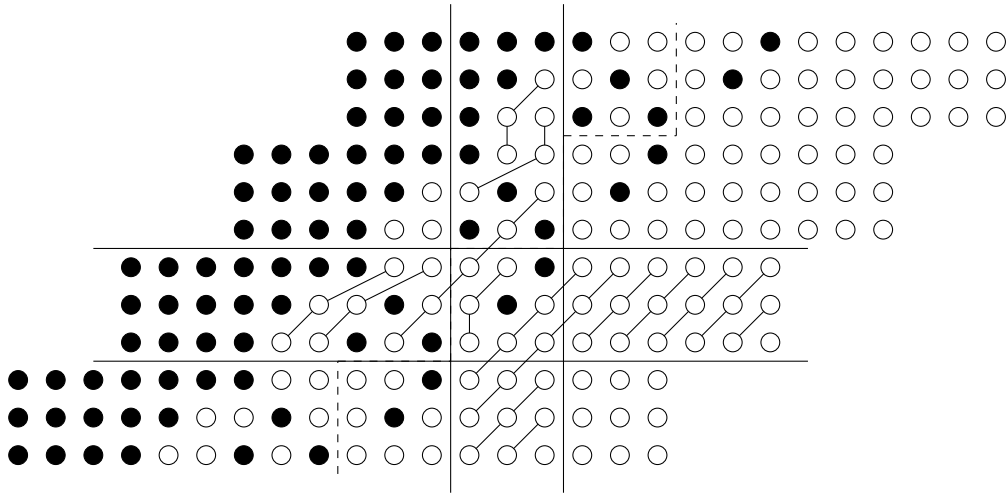

\end{exa}

\begin{rem}\label{tight}
In \cite[Definition 3.8]{Tingley2008}, Tingley defines the notion of a ``tight'' descending abacus.
It is proved in \cite[Proposition 5.15]{Gerber2016a} that Tingley's tightening operators correspond to the raising Heisenberg crystal operators.
Moreover, we already observed in Remark \ref{rem1} (3) that descending abaci correspond to cylindric multipartitions.
Therefore by Theorem \ref{thm_main} (2), tight descending abaci correspond to FLOTW multipartitions.
This shall come as no surprise because Tingley's tight descending abaci are constructed to realize irreducible highest weight crystals \cite[Theorem 3.14]{Tingley2008},
and so are FLOTW multipartitions (as mentioned in Remark \ref{rem_flotw}).
\end{rem}

\pagebreak

\bibliographystyle{plain}

\begin{thebibliography}{10}

\bibitem{Ariki1996}
Susumu Ariki.
\newblock {On the decomposition numbers of the {H}ecke algebra of
  {$G(m,1,n)$}}.
\newblock {\em J. Math. Kyoto Univ.}, 36(4):789--808, 1996.

\bibitem{Ariki2007}
Susumu Ariki.
\newblock {Proof of the modular branching rule for cyclotomic Hecke algebras}.
\newblock {\em J. Alg.}, 306:290--300, 2007.

\bibitem{Borodin2007}
Alexei Borodin.
\newblock {Periodic Schur process and cylindric partitions}.
\newblock {\em Duke Math. J.}, 140:391--468, 2007.

\bibitem{FLOTW1999}
Omar Foda, Bernard Leclerc, Masato Okado, Jean-Yves Thibon, and Trevor Welsh.
\newblock {Branching functions of $A_{n-1}^{(1)}$ and Jantzen-Seitz problem for
  Ariki-Koike algebras}.
\newblock {\em Adv. Math.}, 141:322--365, 1999.

\bibitem{FodaWelsh2015}
Omar Foda and Trevor Welsh.
\newblock {Cylindric partitions, $W_r$ characters and the
  Andrews-Gordon-Bressoud identities}.
\newblock {\em J. Phys. A: Math. Theor.}, 49:164004, 2016.

\bibitem{GeckJacon2006}
Meinolf Geck and Nicolas Jacon.
\newblock {Canonical basic sets in type $B$}.
\newblock {\em J. Alg.}, 306:104--127, 2006.

\bibitem{GeckJacon2011}
Meinolf Geck and Nicolas Jacon.
\newblock {\em {Representations of Hecke Algebras at Roots of Unity}}.
\newblock Springer, 2011.

\bibitem{Gerber2016a}
Thomas Gerber.
\newblock {Heisenberg algebra, wedges and crystals}.
\newblock {\em J. Alg. Comb.}, 2018.
\newblock https://doi.org/10.1007/s10801-018-0820-8.

\bibitem{Gerber2016}
Thomas Gerber.
\newblock {Triple crystal action in Fock spaces}.
\newblock {\em Adv. Math.}, 329:916--954, 2018.

\bibitem{GerberNorton2017}
Thomas Gerber and Emily Norton.
\newblock {The $\mathfrak{sl}_\infty$-crystal combinatorics of higher level
  Fock spaces}.
\newblock {\em J. Comb. Alg.}, 2:1--43, 2018.

\bibitem{GesselKrattenthaler1997}
Ira Gessel and Christian Krattenthaler.
\newblock {Cylindric partitions}.
\newblock {\em Trans. Amer. Math. Soc.}, 349:429--479, 1997.

\bibitem{Jacon2004}
Nicolas Jacon.
\newblock {On the parametrization of the simple modules for Ariki-Koike
  algebras at roots of unity}.
\newblock {\em J. Math. Kyoto Univ.}, 44:729--767, 2004.

\bibitem{Jacon2017}
Nicolas Jacon.
\newblock {Kleshchev multipartitions and extended Young diagrams}.
\newblock 2017.
\newblock arXiv:1706.07595.

\bibitem{JaconLecouvey2012}
Nicolas Jacon and C\'edric Lecouvey.
\newblock {A combinatorial decomposition of higher level Fock spaces}.
\newblock {\em Osaka J. Math.}, 50(4):897--920, 2013.

\bibitem{JMMO1991}
Michio Jimbo, Kailash~C. Misra, Tetsuji Miwa, and Masato Okado.
\newblock Combinatorics of representations of {$U_q(\widehat{sl(n)})$} at
  {$q=0$}.
\newblock {\em Comm. Math. Phys.}, 136(3):543--566, 1991.

\bibitem{Losev2015}
Ivan Losev.
\newblock {Supports of simple modules in cyclotomic Cherednik categories O}.
\newblock 2015.
\newblock arXiv:1509.00526.

\bibitem{Shan2011}
Peng Shan.
\newblock {Crystals of Fock spaces and cyclotomic rational double affine Hecke
  algebras}.
\newblock {\em Ann. Sci. \'Ec. Norm. Sup\'er.}, 44:147--182, 2011.

\bibitem{ShanVasserot2012}
Peng Shan and \'Eric Vasserot.
\newblock {Heisenberg algebras and rational double affine Hecke algebras}.
\newblock {\em J. Amer. Math. Soc.}, 25:959--1031, 2012.

\bibitem{Tingley2008}
Peter Tingley.
\newblock {Three combinatorial models for $\widehat{\mathfrak{sl}_n}$ crystals,
  with applications to cylindric plane partitions}.
\newblock {\em Int. Math. Res. Notices}, Art. ID RNM143:1--40, 2008.

\bibitem{Uglov1999}
Denis Uglov.
\newblock {Canonical bases of higher-level $q$-deformed Fock spaces and
  Kazhdan-Lusztig polynomials}.
\newblock {\em Progr. Math.}, 191:249--299, 1999.

\end{thebibliography}

\end{document}